\documentclass{amsart}
\usepackage[margin=1.6in]{geometry}

\usepackage{amsfonts}

\usepackage{setspace}
\usepackage{graphicx}

\usepackage{hyperref}
\usepackage{graphicx}
\usepackage{amsmath, amsthm, amscd, amsfonts, amssymb, graphicx, color}
 \usepackage{url}
\usepackage{enumerate}
 \makeatletter
\let\reftagform@=\tagform@
\def\tagform@#1{\maketag@@@{(\ignorespaces\textcolor{blue}{#1}\unskip\@@italiccorr)}}
\renewcommand{\eqref}[1]{\textup{\reftagform@{\ref{#1}}}}
\makeatother
\usepackage{hyperref}
\hypersetup{colorlinks=true, linkcolor=red, anchorcolor=green,
citecolor=cyan, urlcolor=red, filecolor=magenta, pdftoolbar=true}

\usepackage{epsfig}        
\usepackage{epic,eepic}       

\newtheorem{theorem}{Theorem}
\theoremstyle{plain}

\newtheorem{corollary}{Corollary}

\newtheorem{lemma}{Lemma}

\newtheorem{proposition}{Proposition}
\newtheorem{remark}{Remark}

\numberwithin{equation}{section}

\begin{document}
\title[On $q$-Bernoulli inequality]{On $q$-Bernoulli inequality}

\author[M.W. Alomari]{Mohammad W. Alomari}

\address{Department of Mathematics, Faculty of Science and
Information Technology, Irbid National University, 2600 Irbid
21110, Jordan.} \email{mwomath@gmail.com}
\date{\today}


\begin{abstract}
In this work, the $q$--analogue of Bernoulli inequality is proved.
Some other related results are presented.
\end{abstract}

\maketitle
\section{Introduction}
Throughout this work, we consider $q\in(0,1)$. The $q$-number is
defined to be the number of the form
\begin{align*}
\left[ \alpha  \right]_q  = \frac{{q^\alpha   - 1}}{{q - 1}},
\qquad \text{for any } \alpha \in \mathbb{C}.
\end{align*}
In particular, if $\alpha =n\in \mathbb{N}$, then the positive
$q$-integer is defined to be
\begin{align*}
\left[n \right]_q  = \frac{{q^n   - 1}}{{q - 1}} =1+q+q^2+\cdots
q^{n-1}.
\end{align*}
In special case, we have $\left[1 \right]_q=1$ and $\left[0
\right]_q=\frac{1}{1-q}=\left[\infty \right]_q$.

We define the $q$-factorial of the number $[n]_q$ and the
$g$-binomial coefficient by
\begin{align*}
\left[ 0 \right]_q ! &= 1, \qquad \left[ n \right]_q ! = \left[ n
\right]_q \left[ {n - 1} \right]_q  \cdots \left[ 2 \right]_q
\cdot \left[ 1 \right]_q \qquad   \left[ {\begin{array}{*{20}c}
   n  \\
   j  \\
\end{array}} \right]_q  = \frac{{\left[ n \right]_q !}}{{\left[ j \right]_q !\left[ {n - j} \right]_q !}}
\end{align*}
with the convention that
\begin{align*}
\left[ {\begin{array}{*{20}c}
   0   \\
   0  \\
\end{array}} \right]_q&=1,\qquad \left[
{\begin{array}{*{20}c}
   0   \\
   j  \\
\end{array}} \right]_q=0, \,\,\forall j\ge1.
\end{align*}
The $q$-Pochammer symbol is defined to be
\begin{align*}
\left( {x - a} \right)_q^n =   \prod\limits_{j = 0}^{n - 1}
{\left( {x - q^j a} \right)},\,\,\,\text{with}\,\, \left( {x - a}
\right)_q^{(0)}=1,\,\,\,\text{and}\,\,  \left( {x - a} \right)_q^{
- n} = \frac{1}{{\left( {x - q^{ - n} a} \right)_q^n }}.
\end{align*}
This formula plays an important role in combinatorics. For
instance, for $x=1$ and $x=a$ this formula make sense as
$n=\infty$:
\begin{align*}
\left( {1 + x} \right)_q^n =   \prod\limits_{j = 0}^{\infty}
{\left( {1 + q^j x} \right)}.
\end{align*}
The above infinite product converges if $q\in
\left(0,\infty\right)$.

We adopt  the symbol
\begin{align*}
\left( {1 + x} \right)_q^\alpha   = \frac{{\left( {1 + x}
\right)_q^\infty }}{{\left( {1 + q^\alpha  x} \right)_q^\infty }}
\end{align*}
for any number $\alpha$. Clearly, this definition coincides with
definition of $\left( {1 + x} \right)_q^n $ when $\alpha=n\in
\mathbb{N}$.
\begin{lemma}\cite{K}\label{lem1}
 For any two numbers $\alpha$ and $\beta$, we have
\begin{align*}
\left( {1 + x} \right)_q^\alpha  =\frac{ \left( {1 + x}
\right)_q^{\alpha  + \beta }}{ \left( {1 + q^\alpha  x}
\right)_q^\beta  }
\end{align*}
and
\begin{align*}
D_q \left( {1 + x} \right)_q^\alpha   = \left[ \alpha  \right]_q
\left( {1 + qx} \right)_q^{\alpha  - 1}.
\end{align*}
\end{lemma}

The $q$-derivative of any real valued function $f$ is defined to
be
\begin{align*}
D_q f\left( x \right) = \frac{{f\left( {qx} \right) - f\left( x
\right)}}{{\left( {q - 1} \right)x}}, \qquad x\ne0.
\end{align*}
Clearly, as $q\to 1^-$ then $D_q f\left( x \right)$ tends to
$f^{\prime}(x)$, provided that $f$ is differentiable.

Two fundamentals $q$-binomial formulas are well know in
Literature. The $q$-Gauss binomial which has the form
\begin{align*}
\left( {1 + x} \right)_q^n   = \sum\limits_{j = 0}^n {\left[
{\begin{array}{*{20}c}
   n   \\
   j  \\
\end{array}} \right]_q q^{j\left( {j - 1} \right)/2} x^j }
\end{align*}
and the  $q$-Heine's binomial formula
\begin{align*}
\frac{1}{\left( {1 - x} \right)_q^n}   = \sum\limits_{j =
0}^\infty {\left[ {\begin{array}{*{20}c}
   n  \\
   j  \\
\end{array}} \right]_q   x^j }.
\end{align*}
However, since
\begin{align*}
\mathop {\lim }\limits_{n \to \infty } \left[
{\begin{array}{*{20}c}
   n  \\
   j  \\
\end{array}} \right]_q  =
  \frac{1}{{\left( {1 - q} \right)\left( {1 - q^2 } \right) \ldots
\left( {1 - q^j } \right)}}
\end{align*}
Applying this for $q$-Gauss and $q$-Heine's binomial formulas, we
get two  formal power series in $x$. Namely, we have
\begin{align}
\left( {1 + x} \right)_q^\infty  = \sum\limits_{j = 0}^\infty {
q^{j\left( {j - 1} \right)/2} \frac{{x^j }}{{\left( {1 - q}
\right)\left( {1 - q^2 } \right) \ldots \left( {1 - q^j }
\right)}}  },\label{eq1.1}
\end{align}
and
\begin{align}
\frac{1}{\left( {1 - x} \right)_q^\infty}   = \sum\limits_{j =
0}^\infty { \frac{{x^j }}{{\left( {1 - q} \right)\left( {1 - q^2 }
\right) \ldots \left( {1 - q^j } \right)}} }.\label{eq1.2}
\end{align}
These two series are very useful in the theory of $q$-calculus,
since they were used to define the $q$-analogue of exponential
function. From \eqref{eq1.2}
\begin{align*}
\frac{1}{\left( {1 - x} \right)_q^\infty} =\sum\limits_{j = 0}^n
{\frac{{\left( {{\textstyle{x \over {1 - q}}}} \right)^j
}}{{\left( {{\textstyle{{1 - q} \over {1 - q}}}} \right)\left(
{{\textstyle{{1 - q^2 } \over {1 - q}}}} \right) \cdots \left(
{{\textstyle{{1 - q^j } \over {1 - q}}}} \right)}}} =
\sum\limits_{j = 0}^n {\frac{{\left( {{\textstyle{x \over {1 -
q}}}} \right)^j }}{{\left[ 1 \right]_1 \left[ 2 \right]_q  \cdots
\left[ j \right]_q }}}  = \sum\limits_{j = 0}^n {\frac{{\left(
{{\textstyle{x \over {1 - q}}}} \right)^j }}{{\left[ j \right]_q
!}} = {\rm{e}}_q^{{\textstyle{x \over {1 - q}}}} },
\end{align*}
or we write
\begin{align*}
{{\rm{e}}_q^x } =\frac{1}{\left( {1 - \left(1-q\right)x}
\right)_q^\infty}.
\end{align*}
Similarly, if we use \eqref{eq1.1} the companion $q$-exponential
function is defined to be
\begin{align*}
\left( {1 + x} \right)_q^\infty  =\sum\limits_{j = 0}^n
{\frac{{q^{j\left( {j - 1} \right)/2} \left( {{\textstyle{x \over
{1 - q}}}} \right)^j }}{{\left( {{\textstyle{{1 - q} \over {1 -
q}}}} \right)\left( {{\textstyle{{1 - q^2 } \over {1 - q}}}}
\right) \cdots \left( {{\textstyle{{1 - q^j } \over {1 - q}}}}
\right)}}}  = \sum\limits_{j = 0}^n {\frac{{q^{j\left( {j - 1}
\right)/2} \left( {{\textstyle{x \over {1 - q}}}} \right)^j
}}{{\left[ 1 \right]_1 \left[ 2 \right]_q  \cdots \left[ j
\right]_q }}}  = \sum\limits_{j = 0}^n {\frac{{\left(
{{\textstyle{x \over {1 - q}}}} \right)^j }}{{\left[ j \right]_q
!}} = {\rm{E}}_q^{{\textstyle{x \over {1 - q}}}} }
\end{align*}
or we write
\begin{align*}
{{\rm{E}}_q^x } = \left( {1 + \left(1-q\right)x} \right)_q^\infty.
\end{align*}
The derivatives of the  above two $q$-exponential functions are
given as
\begin{align*}  D_q{{\rm{E}}_q^x } ={{\rm{E}}_q^{qx} },\qquad
 D_q {{\rm{e}}_q^x }={{\rm{e}}_q^x }
\end{align*}
We note that the additive property of the exponentials does not
hold in general, i.e.,
\begin{align*}
{\rm{e}}_q^x {\rm{e}}_q^y  = {\rm{e}}_q^{x + y}.
\end{align*}
However, if $x$ and $y$ satisfy the commutation relation $yx=qxy$,
then the additive property holds.

The two functions ${\rm{E}}_q^x$ and ${\rm{e}}_q^x$  are connected
to each other by the relations
\begin{align*}
{\rm{E}}_q^{ - x} {\rm{e}}_q^x  = 1 ,\qquad  {\rm{e}}_{1/q}^x =
{\rm{E}}_q^x.
\end{align*}

Naturally, it is important to know the relation between these
$q$-quantities. One of the most  effective method  is to use
 inequalities. Among others, one of the most famous and
applicable inequalities used in mathematics is the Bernoulli
inequality, which is well known as:
\begin{align}
\left( {1 + x} \right)^n   \ge 1 + n x \label{eq1.3}
\end{align}
for every  $x>-1$ and every positive integer $n\ge1$. This was
extended  to more general form such as \cite{MPF}:
\begin{align}
\left( {1 + x} \right)^\alpha   \ge 1 +\alpha x, \qquad \alpha \ge
1\label{eq1.4}
\end{align}
and
\begin{align}
\left( {1 + x} \right)^\alpha   \le 1 + \alpha x, \qquad 0<\alpha
< 1.\label{eq1.5}
\end{align}
This inequality has important applications in proving some
classical theorems in Analysis and Statistics. Due to its
important role, in this work  we prove the $q$-analogue of
Bernoulli inequality and give some other related inequalities.

\section{ $q$-Bernoulli inequality}
Let us begin with the following version of $q$-Bernoulli
inequality for integers.
\begin{theorem}\label{thm1}
Let $q\in (0,1)$. If  $x>-1$ then the $q$-Bernoulli inequality
\begin{align}
\left( {1 + x} \right)_q^n   \ge 1 + \left[ n  \right]_q
x,\label{eq2.1}
\end{align}
is valid for every positive integer $n\ge1$.
\end{theorem}
\begin{proof}
Our proof carries by Induction. Define the statement
\begin{align}
{\rm{P}}\left(n\right): \qquad\left( {1 + x} \right)_q^n   \ge 1 +
\left[ n \right]_q x.\label{eq2.2}
\end{align}
If $x=0$, then the we get equality for all $n$ and thus
\eqref{eq2.1} holds.

If   $x>-1$. Then,
\begin{align*}
{\rm{P}}\left(1\right): \qquad\left( {1 + x} \right)^1_q=  \left(
{1 + x} \right) \ge  1 +
 x=1 + \left[ 1 \right]_q x
\end{align*}
Assume \eqref{eq2.2} holds for $n=k$, i.e.,
\begin{align*}
{\rm{P}}\left(k\right): \qquad\left( {1 + x} \right)_q^k   \ge 1 +
\left[ k \right]_q x  \qquad\text{is true}.
\end{align*}
We need to show that
\begin{align*}
{\rm{P}}\left(k+1\right): \qquad\left( {1 + x} \right)_q^{k+1} \ge
1 + \left[ k+1 \right]_q x
\end{align*}
is true?.

Starting with the left-hand side
\begin{align*}
\left( {1 + x} \right)_q^{k+1}&=\left( {1 + x} \right)_q^k\left(
{1 + q^kx} \right) \\&\ge \left(1 + \left[ k \right]_q
x\right)\left( {1 + q^kx} \right)\qquad \qquad\qquad\text{(follows
by assumption \eqref{eq2.2} for $n=k$)}
\\
&=1 + \left[ k \right]_q x + q^k x + q^k \left[ k \right]_q x^2  \\
  &\ge 1 + \frac{1}{{q  }}q  \left[ k \right]_q x + \frac{q}{q}q^k x    \\
  &= 1 + \frac{1}{{q  }}\left( {\left[ {k + 1} \right]_q  - 1} \right)x + \frac{1}{q}\left( {q^{k + 1} x} \right) \qquad \qquad\text{(since $
\left[ {k + 1} \right]_q  = q  \left[ k \right]_q  + 1$)}   \\
  &= 1 + \frac{1}{q}\left( {\left[ {k + 1} \right]_q  + q^{k + 1}  - 1} \right)x \\
  &= 1 + \frac{1}{q}\left( {\left[ {k + 1} \right]_q  + \left( {q - 1} \right)\left[ {k + 1} \right]_q } \right)x\qquad\qquad \text{(since $
q^{k + 1}  - 1 = \left( {q - 1} \right)\left[ {k + 1} \right]_q$)} \\
  &= 1 + \left[ {k + 1} \right]_q x,
\end{align*}
which means the statement ${\rm{P}}\left(k+1\right)$ is true and
thus by Mathematical
 Induction hypothesis the inequality in \eqref{eq2.1} holds for every $n\in
 \mathbb{N}$ and $x>-1$.
\end{proof}

\begin{remark}
\label{rem1}As $q\to 1$ in \eqref{eq2.1}, then the $q$-Bernoulli
inequality \eqref{eq2.1} reduces to the original version of
Bernoulli inequality \eqref{eq1.3} for integer case.
\end{remark}
\begin{remark}
For the case $-1<x<0$, we prefer to write \eqref{eq2.1}  in the
form
\begin{align*} \left( {1 - y} \right)_q^n \ge 1 -
\left[n\right]_qy
\end{align*}
 for every $0< y<1$ and $n \ge 1$.\label{rem1}
\end{remark}

\begin{corollary}\label{cor1}
Let $q\in (0,1)$. If  $x>-1$, then the generalization
$q$-Bernoulli inequality
\begin{align*}
\left( {1 + x} \right)_q^{m + n}   \ge \left(1 + \left[ m
\right]_q x\right) \left( {1 + q^m x} \right)_q^n,
\end{align*}
is valid for every   $m\in \mathbb{N}$ and $n\in \mathbb{Z}$.
\end{corollary}
\begin{proof}
The result  is an immediate consequence of Theorem \ref{thm1}, by
substituting  $\left( {1 + x} \right)_q^m = \frac{{\left( {1 + x}
\right)_q^{m + n} }}{{\left( {1 + q^m x} \right)_q^n }}$ in
\eqref{eq2.1}. So that the result follows for every $m\in
\mathbb{N}$ and $n\in \mathbb{Z}$.
\end{proof}

The following generalization of \eqref{eq2.1} is valid for any
real number $\alpha\ge0$.
\begin{theorem}\label{thm2}
Let $q\in (0,1)$. If  $x\ge0$ then the $q$-Bernoulli inequality
\begin{align}
\left( {1 + x} \right)_q^\alpha   \ge 1 + \left[ \alpha  \right]_q
x, \qquad \alpha \ge 1 \label{eq2.3}
\end{align}
and
\begin{align}
\left( {1 + x} \right)_q^\alpha   \le 1 + \left[ \alpha  \right]_q
x, \qquad 0<\alpha < 1 \label{eq2.4}
\end{align}
is valid.
\end{theorem}

\begin{proof}
Let us recall that \cite{G}, for $0<a<b$  (or $0>a>b$), a function
$f(x)$ is said to be $q$-increasing (respectively, $q$-decreasing)
on $[a,b]$, if $f(qx)\le f(x)$ (respectively, $f(qx)\ge f(x)$)
whenever, $x\in [a,b]$ and $qx \in [a,b]$. As a direct consequence
we have, $f(x)$ is $q$-increasing (respectively, $q$-decreasing)
on $[a,b]$ iff $D_qf(x)\ge 0$ (respectively, $D_qf(x)\le 0$),
whenever, $x\in [a,b]$ and $qx \in [a,b]$.

Let $f\left( x \right) = \left( {1 + x} \right)_q^\alpha   -
\left[ \alpha  \right]_q x - 1$, $x\ge0$. Since $ \left( {1 + x}
\right)_q^\alpha   = \frac{{\left( {1 + x} \right)_q^\infty
}}{{\left( {1 + q^\alpha  x} \right)_q^\infty }}$, inserting $qx$
instead of $x$ and replace $\alpha$ by $\alpha-1$ we get $ \left(
{1 + qx} \right)_q^{\alpha  - 1}  = \frac{{\left( {1 + qx}
\right)_q^\infty  }}{{\left( {1 + q^\alpha  x} \right)_q^\infty }}
$. Therefore, we have
\begin{align*}
D_q f\left( x \right)  &= \left[ \alpha  \right]_q \left( {1 + qx}
\right)_q^{\alpha  - 1}  - \left[ \alpha  \right]_q\\
&= \left[ \alpha  \right]_q \frac{{\left( {1 + qx} \right)_q^\infty  }}{{\left( {1 + q^\alpha  x} \right)_q^\infty  }} - \left[ \alpha  \right]_q\\
  &= \left[ \alpha  \right]_q \frac{{\sum\limits_{j = 0}^\infty  {{\textstyle{{q^{j\left( {j - 1} \right)/2} } \over {\prod\limits_{k = 1}^j {\left( {1 - q^k } \right)} }}}q^j x^j } }}{{\sum\limits_{j = 0}^\infty  {{\textstyle{{q^{j\left( {j - 1} \right)/2} } \over {\prod\limits_{k = 1}^j {\left( {1 - q^k } \right)} }}}q^{j\alpha } x^j } }} - \left[ \alpha
  \right]_q\qquad (\text{with the convention} \prod\limits_{k = 1}^0 {\left( {1 - q^j } \right)}  = 1) \\
  &= \left[ \alpha  \right]_q \frac{{1 + \sum\limits_{j = 1}^\infty  {{\textstyle{{q^{j\left( {j - 1} \right)/2} } \over {\prod\limits_{k = 1}^j {\left( {1 - q^k } \right)} }}}q^j x^j } }}{{1 + \sum\limits_{j = 1}^\infty  {{\textstyle{{q^{j\left( {j - 1} \right)/2} } \over {\prod\limits_{k = 1}^j {\left( {1 - q^k } \right)} }}}q^{j\alpha } x^j } }} - \left[ \alpha
  \right]_q\\
 &= \left[ \alpha  \right]_q \frac{{\sum\limits_{j = 1}^\infty
{{\textstyle{{q^{j\left( {j - 1} \right)/2} } \over
{\prod\limits_{k = 1}^j {\left( {1 - q^k } \right)} }}}\left( {q^j
- q^{j\alpha } } \right)x^j } }}{{1 + \sum\limits_{j = 1}^\infty
{{\textstyle{{q^{j\left( {j - 1} \right)/2} } \over
{\prod\limits_{k = 1}^j {\left( {1 - q^k } \right)} }}}q^{j\left(
{\alpha  - 1} \right)} x^j } }}
   \ge 0,
\end{align*}
since $q\in (0,1)$ and  $\alpha \ge 1$ then $\left( {q^j  -
q^{j\alpha } } \right)>0$,
 and this implies that $D_q f\left( x \right) \ge 0$ for
all $x\ge0$, which means that $f$ is $q$-increasing and thus the
inequality \eqref{eq2.3} is proved.

The inequality \eqref{eq2.4} is deduced from the above proof by
noting that   $\left( {q^j  - q^{j\alpha } } \right)<0$ for all
$0<\alpha< 1$.
\end{proof}

\begin{remark}\label{rem2}
Setting $\alpha =n\in \mathbb{N}$ in \eqref{eq2.3}, then  the
inequality \eqref{eq2.3} reduces to the  $q$-version of Bernoulli
inequality \eqref{eq2.1} for integer case but for $x\ge0$.
Moreover, as $q\to 1$ \eqref{eq2.3} and \eqref{eq2.4} reduces to
the classical versions \eqref{eq1.4} and \eqref{eq1.5};
respectively.
\end{remark}

Testing the validity of \eqref{eq2.3} and \eqref{eq2.4} for
$-1<x<0$ arbitrarily, we find that these inequalities
 can be extended but with additional restriction on
$q\in (0,1)$, as given in the following result.
\begin{theorem}
\label{thm3}There exists $\widehat{q}\in (0,1)$ such that the
inequalities \eqref{eq2.3} and \eqref{eq2.4} are hold for every
$q\in (\widehat{q},1)$ and every $x>-1$.
\end{theorem}

\begin{proof}
Firstly, we need to recall the $q$-Mean Value Theorem ($q$-MVT)
given in \cite{RSM}, it states that:  For a continuous function
$g$ defined on $[a,b]$ $(0<a<b)$, there exist $\eta\in (a,b)$ and
$\widehat{q}\in (0,1)$ such that
\begin{align}
\label{q-MVT}g\left( {b} \right) - g\left( a \right) = D_q g\left(
\eta \right)\left( {b - a} \right)
\end{align}
for all $q\in \left(\widehat{q},1\right) $.\\

\noindent\textbf{Case I.} If $x\ge0$. We consider the function
$f\left(t\right)=\left(1+t\right)^\alpha_q$ defined for $t\ge0$.
Clearly $f$ is continuous for $t\in [0,x]\subset [0,\infty)$, and
$D_q f\left( c \right) = \left[ \alpha  \right]_q \left( {1 + qc}
\right)_q^{\alpha  - 1}$. Applying, \eqref{q-MVT} for $a=0$ and
$b=x$ then there exist $\eta\in (a,b)$ and $\widehat{q}\in (0,1)$
\begin{align*}
\left(1+x\right)^\alpha_q-1=\left[ \alpha  \right]_q\left( {1 +
q\eta} \right)_q^{\alpha  - 1} \left(x-0\right) \ge \left[ \alpha
\right]_qx\qquad \forall q\in \left(\widehat{q},1\right).
\end{align*}
This yields that
\begin{align*}
\left(1+x\right)^\alpha_q \ge 1+ \left[ \alpha \right]_qx
\end{align*}
$\forall q\in \left(\widehat{q},1\right)$, and this proves
\eqref{eq2.3}.\\

\noindent\textbf{Case II.} If $-1<x<0$. Let us rewrite
\eqref{eq2.3} as follows:
\begin{align}
1 - \left[ \alpha  \right]_q y \le \left( {1 - y}
\right)_q^\alpha, \qquad 0< y < 1\label{eq2.6}
\end{align}
Consider the function $f\left(t\right)=\left(1-t\right)^\alpha_q$
defined for $0\le t\le 1$. Clearly $f$ is continuous for $t\in
[0,y]\subset [0,1]$, and $D_q f\left( c \right) = \left[ \alpha
\right]_q \left( {1 + qc} \right)_q^{\alpha  - 1}$. Applying,
\eqref{q-MVT} for $a=0$ and $b=y$ then there exist $\eta\in (0,x)$
and $\widehat{q}\in (0,1)$
\begin{align}
\left(1-y\right)^\alpha_q-1=-\left[ \alpha  \right]_q\left( {1 -
q\eta} \right)_q^{\alpha  - 1} \left(y-0\right) \ge -\left[ \alpha
\right]_qy\qquad \forall q\in
\left(\widehat{q},1\right).\label{eq2.7}
\end{align}
This yields that
\begin{align*}
\left(1-y\right)^\alpha_q \ge 1- \left[ \alpha \right]_qy
\end{align*}
$\forall q\in \left(\widehat{q},1\right)$ with $0<y<1$, and this
proves the inequality. The reverse inequality in \eqref{eq2.6}
holds since the inequality in \eqref{eq2.7} is reversed for
$0<\alpha<1$, which proves \eqref{eq2.4}
\end{proof}

A generalization of \eqref{eq2.3} and \eqref{eq2.4} is given as
follows:
\begin{proposition}\label{prp1}
Let  $\beta\in \mathbb{R}$. There exists $\widehat{q}\in (0,1)$
such that for every $x>-1$ the inequalities
\begin{align}
\left( {1 + x} \right)_q^{\alpha  + \beta }  \ge \left( {1 +
\left[ \alpha  \right]_q x} \right)\left( {1 + q^\alpha  x}
\right)_q^\beta   \qquad \alpha \ge 1\label{eq2.8}
\end{align}
and
\begin{align}
\left( {1 + x} \right)_q^{\alpha  + \beta }  \le \left( {1 +
\left[ \alpha  \right]_q x} \right)\left( {1 + q^\alpha  x}
\right)_q^\beta   \qquad 0<\alpha < 1\label{eq2.9}
\end{align}
are hold for every $q\in (\widehat{q},1)$.
\end{proposition}
\begin{proof}
From Lemma \ref{lem1} we have $\left( {1 + x} \right)_q^\alpha   =
\frac{{\left( {1 + x} \right)_q^{\alpha  + \beta } }}{{\left( {1 +
q^\alpha  x} \right)_q^\beta  }}$. Substituting in \eqref{eq2.3}
we get the required result.
\end{proof}
\begin{remark}\label{rem3}
Setting $\beta=0$ in \eqref{eq2.8} and \eqref{eq2.9} we recapture
\eqref{eq2.3} and \eqref{eq2.4}, respectively.
\end{remark}

\begin{corollary}\label{cor6}
Let  $\beta\in \mathbb{R}$. There exists $\widehat{q}\in (0,1)$
such that for every $x>-1$ the inequalities
\begin{align}
\left( {1 + x} \right)_q^{\infty  }  \ge \left( {1 + \left[ \alpha
\right]_q x} \right)\left( {1 + q^\alpha  x} \right)_q^\beta\left(
{1 + q^{\alpha+\beta} x} \right)_q^\infty \qquad \alpha \ge
1\label{eq2.10}
\end{align}
and
\begin{align}
\left( {1 + x} \right)_q^{\infty }  \le \left( {1 + \left[ \alpha
\right]_q x} \right)\left( {1 + q^\alpha  x} \right)_q^\beta\left(
{1 + q^{\alpha+\beta} x} \right)_q^\infty \qquad 0<\alpha <
1\label{eq2.11}
\end{align}
are hold for every $q\in (\widehat{q},1)$
\end{corollary}
\begin{proof}
Substituting $\left( {1 + x} \right)_q^{\alpha+\beta}
=\frac{{\left( {1 + x} \right)_q^\infty  }}{{\left( {1 +
q^{\alpha+\beta} x} \right)_q^\infty }}$ in \eqref{eq2.8} and
\eqref{eq2.9}; respectively, we get the required result.
\end{proof}

 \begin{remark}
Replacing $`(1-q)x$' instead of $`x$' in \eqref{eq2.10} and
\eqref{eq2.11}, we get inequalities for the exponential function
${\rm{E}}_q^x$ for all $x>\frac{-1}{1-q}$. Similarly, for
${\rm{e}}_q^x$ with a bit changes in the substitution.
 \end{remark}

\begin{corollary}
There exists $\widehat{q}\in (0,1)$ such that for every $x>-1$ the
inequalities
\begin{align}
\left( {1 + x} \right)_q^{\infty  }  \ge \left( {1 + \left[ \alpha
\right]_q x} \right) \left( {1 + q^{\alpha} x} \right)_q^\infty
\qquad \alpha \ge 1\label{eq2.12}
\end{align}
and
\begin{align}
\left( {1 + x} \right)_q^{\infty }  \le \left( {1 + \left[ \alpha
\right]_q x} \right)\left( {1 + q^{\alpha} x} \right)_q^\infty
\qquad 0<\alpha < 1\label{eq2.13}
\end{align}
are hold for every $q\in (\widehat{q},1)$
\end{corollary}
\begin{proof}
Setting $\beta =0$ in \eqref{eq2.10} and \eqref{eq2.11};
respectively, we get the required result.
\end{proof}

\end{document}